\documentclass[12pt]{article}
\usepackage{amsmath,amsfonts,amssymb,amsthm,url,array,mathtools}
\usepackage[margin=1in]{geometry}
\usepackage{xcolor}
\usepackage{graphicx}
\usepackage{times}
\usepackage{multirow}
\usepackage{authblk}
\graphicspath{{Figures/}}
  
\theoremstyle{plain}
\newtheorem{thm}{Theorem}[section]
\newtheorem{lem}[thm]{Lemma}
\newtheorem{prop}[thm]{Proposition}
\newtheorem{cor}[thm]{Corollary}
\newtheorem{defn}[thm]{Definition}
\newtheorem{claim}[thm]{Claim}

\theoremstyle{definition}
\newtheorem{con}{Conjecture}[section]

\newtheorem{cdef}{Definition}[section]

\theoremstyle{remark}
\newtheorem{rem}{\textbf{Remark}}

\usepackage{qtree}

\usepackage{hyperref}

\usepackage[nameinlink,noabbrev,capitalize]{cleveref} 
\crefalias{subequation}{equation}
\crefalias{thm}{theorem}

\let\oldlemma\lem
\renewcommand{\lem}{%
  \crefalias{thm}{lem}
  \oldlemma}
\Crefname{lem}{Lemma}{Lemmas}

\let\olddefn\defn
\renewcommand{\defn}{%
  \crefalias{thm}{defn}
  \olddefn}
\Crefname{defn}{Definition}{Definitions}

\let\oldcdef\cdef
\renewcommand{\cdef}{%
  \crefalias{thm}{cdef}
  \oldcdef}
\Crefname{cdef}{Definition}{Definitions}

\let\oldrem\rem
\renewcommand{\rem}{%
  \crefalias{thm}{rem}
  \oldrem}
\Crefname{rem}{Remark}{Remarks}

\let\oldcor\cor
\renewcommand{\cor}{%
  \crefalias{thm}{cor}
  \oldcor}
\Crefname{cor}{Corollary}{Corollaries}

\let\oldclaim\claim
\renewcommand{\claim}{%
  \crefalias{thm}{claim}
  \oldclaim}
\Crefname{claim}{Claim}{Claims}

\let\oldprop\prop
\renewcommand{\prop}{%
  \crefalias{thm}{prop}
  \oldprop}
\Crefname{prop}{Proposition}{Propositions}

\let\oldcon\con
\renewcommand{\con}{%
  \crefalias{thm}{con}
  \oldcon}
\Crefname{con}{Conjecture}{Conjectures}

\definecolor{darkgrn}{rgb}{0, 0.8, 0}

\definecolor{maroon}{rgb}{0.85, 0.0, 0.1}

\newcommand{\vF}{ F }
\newcommand{\vb}{ \mathbf{b} }
\newcommand{\vB}{ \mathbf{B} }
\newcommand{\ve}{ \mathbf{e} }
\newcommand{\vl}{ \mathbf{l} }
\newcommand{\vu}{ \mathbf{u} }
\newcommand{\vx}{ \mathbf{x} }
\newcommand{\vy}{ \mathbf{y} }
\newcommand{\vz}{ \mathbf{z} }
\newcommand{\vzero}{ \mathbf{0} }
\newcommand{\vone}{ \mathbf{1} }
\newcommand{\vlambda}{ \boldsymbol{\lambda} }
\newcommand{\rz}{\mathrm{z}}
\newcommand{\lam}{\lambda}

\newcommand{\R}{\mathbb{R}}
\newcommand{\Int}{\operatorname{Int}}
\newcommand{\Tr}{\operatorname{Trace}}
\newcommand{\rank}{\operatorname{rank}}
\DeclarePairedDelimiter{\norm}{\lVert}{\rVert}
\newcommand{\p}{\phantom{-}}

\renewcommand{\Re}{\operatorname{Re}}
\renewcommand{\Im}{\operatorname{Im}}

\title{\vspace*{-0.45in}
  Robust Feasibility of Systems of Quadratic Equations Using Topological Degree Theory}

\author[1]{Krishnamurthy Dvijotham}
\author[2]{Bala Krishnamoorthy\thanks{kbala@wsu.edu}}
\author[3]{Yunqi Luo}
\author[4]{Benjamin Rapone\thanks{Authors are listed alphabetically.}}
\affil[1]{Google DeepMind, USA}
\affil[2]{Wahington State University, USA}
\affil[3]{International Institute of Finance, School of Management, University of Science and Technology of China}
\affil[4]{Washington State Governor's Office, USA}

\date{}

\usepackage{csquotes}

\begin{document}

\maketitle

\vspace*{-0.3in}
\begin{abstract}
  We consider the problem of measuring the \emph{margin of robust feasibility} of solutions to a system of nonlinear equations.
  We study the special case of a system of quadratic equations, which shows up in many practical applications such as the power grid and other infrastructure networks. 
  This problem is a generalization of quadratically constrained quadratic programming (QCQP), which is NP-Hard in the general setting.
  We develop approaches based on topological degree theory to estimate bounds on the robustness margin of such systems.
  Our methods use tools from convex analysis and optimization theory to cast the problems of checking the conditions for robust feasibility as a nonlinear optimization problem.
  We then develop \emph{inner bound} and \emph{outer bound} procedures for this optimization problem, which could be solved efficiently to derive lower and upper bounds, respectively, for the margin of robust feasibility.
  We evaluate our approach numerically on standard instances taken from the MATPOWER and NESTA databases of AC power flow equations that describe the steady state of the power grid.
  The results demonstrate that our approach can produce tight lower and upper bounds on the margin of robust feasibility for such instances.
\end{abstract}

\section{Introduction} \label{sec:intro}  

Solving systems of equations is ubiquitous in computational mathematics.
In many applications, these problems are made challenging due to the functions in the equations being nonlinear and/or nonconvex.
Another aspect adding to the problem complexity is the uncertainty in the problem parameters.
Our work is motivated by two central computations performed as part of power systems operations are power flow (PF) studies and optimal power flow (OPF).
PF studies ensure the power grid state (i.e., voltages and flows across the network) will remain within acceptable limits in spite of contingencies (e.g., loss of a generator or transmission line) and other uncertainties (e.g., shifting demand or renewable sources of power).
OPF seeks further to choose values for controllable assets in the system (e.g., generators whose rate of power production could be controlled) so as to meet demand at minimum cost.
These problems have inherent nonlinearities and nonconvexities, making them hard to solve in their general form.
  
To further complicate the problem, the rapid adoption of renewable energy sources such as wind and solar energy is adding unprecedented uncertainties to modern power systems.
Since these sources depend on the weather, their energy output is not perfectly controllable.
In fact, this output can be forecasted with only limited accuracy.
While demand-side flexibility can be used to balance fluctuations in solar and wind generation, its amount can in turn be difficult to predict \cite{mathieu2011examining,taylor2015uncertainty}.
Due to all these uncertainties, it is increasingly difficult to ensure there is sufficient power generation to meet demand while accounting for losses and network limits.

\medskip
We study quadratic systems of equations with parameters, and take a \emph{robust viewpoint} of uncertainty.
Specifically, we aim to quantify the worst-case impact of uncertainty in parameters on feasibility.
To this end, we study the \emph{robust feasibility problem}, which includes the robust version of the standard PF problem as a special case.
The power system can be described by a system of nonlinear equations in a set of variables that capture the state of the power grid, i.e., voltages at every point in the power network, and include the controllable inputs as well as uncertain inputs.
In the main PF problem, we are given a fixed value of the controllable inputs and an uncertainty set for the uncertain inputs.
The goal of the robust feasibility problem is to characterize whether the system has a solution within specified bounds (capturing engineering limits on voltages, flows, etc.) for {\em each} choice of the uncertain inputs in the uncertainty set.

\medskip
More concretely, we study a system of quadratic equations $F(\vx)=\vu$ where $F: \R^n \mapsto \R^n$ is quadratic in $\vx$ for $\vx,\vu \in \R^n$.
  We consider situations where the parameters $\vu$ are uncertain, and we are interested in guaranteeing the existence of a solution to $F(\vx) = \vu$ within limits on $\vx$ and $\vu$.
We draw on results from topological degree theory and Borsuk's theorem from algebraic topology and nonlinear analysis to develop tests for existence of solutions.
Using ideas from optimization such as convex relaxations of quadratic constraints, we develop rigorous and efficient algorithms based on these tests for robust feasibility.
We develop efficient implementations of these algorithms capable of scalably solving large instances of PF problems.
While we use power systems as the main application area, the methods we develop are fairly general, and could be applied to problems in other domains as well, e.g., stochastic processes and gas distribution networks.

\subsection{Our Contributions}
  We study systems of quadratic equations, and define a \emph{robustness margin} as a measure of the system's robust feasibility (see \cref{RobustDef}).
  We develop approaches based on topological degree theory to estimate bounds on the robustness margin of such systems (see \cref{sec:theory}).
  We use tools from convex analysis and optimization theory to cast the problem of checking the conditions for robust feasibility as a nonlinear optimization problem.
  We then develop \emph{inner bound} (\cref{sec:inbdform}) and \emph{outer bound} (\cref{sec:outbdform}) formulations for this optimization problem, which could be solved efficiently to derive lower and upper bounds, respectively, for the margin of robust feasibility.
  We evaluate our approach numerically on standard instances taken from the MatPower database of AC power flow equations that describe the steady state of the power grid (\cref{sec:numstd}).
  The results demonstrate that our approach can produce tight lower and upper bounds on the robustness margin for such instances.

\subsection{Related Work}

Robust feasibility and optimization have been well-studied by both the optimization and topology communities. 
What is lacking is an approach that can guarantee and quantify robust feasibility on large scale systems in an efficient manner. 
In this article we address this deficiency by developing theory that utilizes results from topological degree theory and convex optimization. 
We provide a theoretical foundation for determining robust feasibility of systems of quadratic equations and computational methods for producing lower and upper bounds on the maximum error bound for which one can guarantee robust solvability (the radius of robust solvability). 
To highlight the efficacy of our approach we derive procedures, which we test numerically on several quadratic systems constructed from the AC power flow equations that describe the steady state of the power grid with added uncertainty. 
The results show that our approach can be applied to large scale systems to produce tight lower and upper bounds on the radius of robust solvability, which we shall define as the robustness margin of the system.

In optimization, the focus has been on robust \emph{convex} optimization where uncertainty sets are specified for the parameters of a convex optimization problem (typically an LP or conic program) \cite{ben2009robust}, while the robust versions of generic polynomial programming problem are related by a hierarchy of SDP relaxations \cite{Lasserre2006,Lasserre2011}.
Robust \emph{nonconvex} optimization has received only limited attention (a notable exception is the work of Bertsimas et al.~\cite{BeNoTe2010}).
These approaches do not provide rigorous guarantees for robust feasibility with nonconvex constraints.

In algebraic topology, there have been a number of studies on these problems based on several approaches, including ones based on robustness of level sets and persistent homology \cite{BeEdMoPa2010,EdMoPa2011}, well groups and diagrams \cite{ChSkPa2012,FrKr2016well,FrKr2016pers}, topological degree and robust satisfiability \cite{FrKr2015,FrKrWa2016},  and on Borsuk's theorem and interval arithmetic \cite{FrRa2015,FrHoLa2007,FrLa2005}.
While the theory developed by these approaches is fairly complete, the associated algorithms typically rely on explicit simplicial or cellular decompositions of the problem space.
But the size of such decompositions typically grows exponentially in the problem dimension, and hence these algorithms are typically impractical for large-scale applications.

Looking specifically at applications such as the power systems, there has been significant interest in solving the non-robust version of the OPF problem to global optimality.
The driver has been the development of strong convex relaxations of the nonconvex optimization problems combined with ideas from global optimization such as spatial branch-and-cut, bound tightening, etc.~\cite{BiMu2016,coffrin2015strengthening}.
Uncertainty has been handled in a chance-constrained framework \cite{BiChHa2014,zhang2011chance}.
However, this approach has typically been applied only to linear approximations or convex relaxations of the AC power flow equations, and does not guarantee feasibility with respect to the true nonlinear power flow equations \cite{BiChHa2014,kocuk2016strong,RoVrOlAn2015,TsBiTa2016}.

There is significant empirical work on solving the PF equations with probabilistic uncertainty \cite{morales2007point,wang1992interval} and specifying conditions on the power injections over which the power flow equations are guaranteed to have a solution \cite{bolognani2016existence,EPFLA,EPFLB}.
However, many of these algorithms are based on sampling heuristics and either do not offer mathematical guarantees of robust feasibility or do not directly address the robust feasibility problem.
More recently, Dvijotham, Nguyen, and Turitsyn \cite{DjTuritsyn} developed an approach to handle uncertainty which produced inner/lower bounds on the distance from the nominal values of the uncertain parameters for which the system can still be guaranteed to have solutions.
This approach closely aligns with the methods describing our inner bound procedures, and further can produce a certificate of tightness under special conditions. 
However, this method depends critically on the choice of norms, which is not straightforward to make.


\paragraph{Notation:}
We denote vectors by bold lowercase letters, e.g., $\vx, \vu \in \R^n$, and matrices by upper case letters, e.g., $A \in \R^{m \times n}$.
The vectors of all zeros and all ones are denoted $\vzero$ and $\vone$, respectively.
Individual entries in a vector $\vx$ are denoted $x_i$, for instance.
We let $[n]$ denote the numbers $\{1,2,\dots,n\}$.

\section{Problem Formulation} \label{sec:probform}  

We study systems of quadratic equations of the form
\begin{align}
& Q(\vx)+L\vx=\vu\label{eq:Quad}
\end{align}
where $Q: \mathbb{R}^n \mapsto \mathbb{R}^n$ is a vector-valued quadratic function, that is, there exist matrices $Q_1,\ldots,Q_n $ $\in \mathbb{R}^{n\times n}$, such that
\[[Q(\vx)]_i = \vx^T Q_{i} \vx \quad \forall i \in [n]\]
and $L \in \mathbb{R}^{n\times n}$, $\vu \in \mathbb{R}^n$. 
We are interested in solutions to this system of equations under linear constraints of the form
\begin{align}
(A\vx)_i\leq b_i \quad \forall i \in [n]\label{eq:xLimits}
\end{align}
where we assume that $(A\vx)_i \leq b_i$ for each $i$ is free of redundant constraints and $\vx$ and $\vu$ have the same dimension.
However, the parameter $\vu$ is uncertain and known only up to certain error bounds:
\begin{align}
u^{\min}_i=u_i^*-e_i \leq u_i \leq u_i^*+e_i=u^{\max}_i \quad \forall i \in [n] \label{eq:uLimits}
\end{align}
where $\vu^*$ is a forecast for $\vu$ and $\ve$ denotes the error bounds associated with the forecast. 
For example, in the case of quadratic equations appearing in infrastructure networks like the power grid, $u_i$ represents uncertain power generation or consumption (for example uncertain weather-dependent power sources like solar or wind power). 
In the case of stochastic processes, $\vu^*$ represents an initial state distribution. Further, note that if the polyhedron given by $A\vx \leq \vb$ is not full dimensional, then there would exist at least a $i\in[1,\dotsb, n]$ such that $x_i$ could be any number as long as satisfies \eqref{eq:Quad}. In this case,for the corresponding uncertain parameter $u_i$, there are no $u_i^{min}$ and $u_i^{\max}$.

\begin{cdef}[Robust Feasibility and Robustness Margin problem]  \label{RobustDef}
  Determine whether for all values of $\vu$ satisfying \eqref{eq:uLimits}, the system of equations \eqref{eq:Quad} has a solution lying within the interior of the set of all  $\vx$ satisfying the constraints in \eqref{eq:xLimits}. 
  If this is true, the system comprised of \eqref{eq:Quad},\eqref{eq:xLimits},\eqref{eq:uLimits} is said to be \emph{robust feasible}. 
  The largest $r$ for which $e_i\geq r \ \forall i~$ with $~e_i>0$ in \eqref{eq:uLimits} and such that the system is robust feasible is defined as the \emph{robustness margin}. 
  See \cref{fig:RobFeas(r)} for a pictorial depiction. 
\end{cdef}

\begin{figure}[htp!]
\begin{center}
  \includegraphics[width=\textwidth]{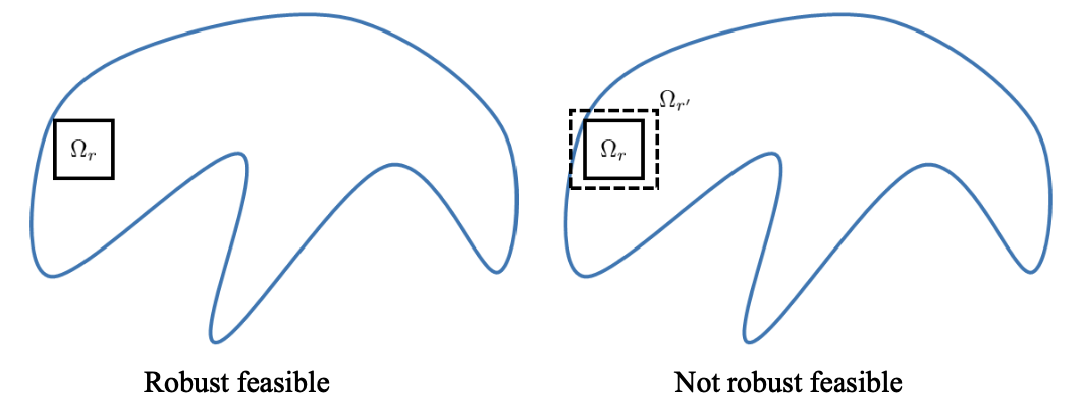}
\end{center}
\caption{Illustration of robust feasibility.
The system is robust feasible at the level of $r$ (left), but is not robust feasible at $r' > r$ (right).}
\label{fig:RobFeas(r)}
\end{figure}

There are many practical uncertainty sets developed in previous studies such as ellipsoidal uncertainty \cite{ben1999robust}, cardinality-constrained uncertainty \cite{bertsimas2004robust}, norm uncertainty \cite{bertsimas2004price}, and some other types of uncertainty sets constructed based on probability theory \cite{bandi2012tractable,bandi2014optimal} and data driven approaches \cite{lotfi2022data}.
  We have presented the definition for robustness margin in the most general form so as to capture most scenarios that may be considered by the researcher.
  Note that for a given choice of $r$, our uncertainty set specified as the $\ell_{\infty}$-ball of radius $r$ forms the largest such set---many other commonly used uncertainty sets, e.g., $\ell_1$-ball, $\ell_2$-ball, ellipsoid (with major axis $r$), are subsets of the $\ell_\infty$-ball.
For instance it very well may be the case that only some of the dimensions of $\vu$ will have margins of uncertainty. Furthermore, one may have need of computing the robustness margin for only a subset of the dimensions of $\vu$ which pertain to problem areas or nodes of particular interest to the research.This manual restriction will of course produce a robustness margin greater than or equal to that obtained by considering all dimensions, which certainly remains an option under the current setting.

\section{Theoretical Results} \label{sec:theory}  
We now describe the main technical results of this paper. 
In the first subsection we describe the setting under which the problem can be solved using the results which follow. 
\subsection{Topological Degree Theory}
Our results take advantage of the well studied area of topological degree theory. 
For an introduction to topological degree theory see the works of \cite{fonseca1995degree}, \cite{MoVrYa2002}, and \cite{OrChCh2006}.
It suffices to say that should $\Omega\subset\R^{n}$ be open and bounded, $F:\Omega\rightarrow \R$ continuous, differentiable, and $F(\vx) \neq \vy~\forall \vx\in\partial\Omega$ for some $\vy\in\R^n$, then the degree of $F$ at $\vy$ over $\Omega$, denoted $d\left(\Omega,F,\vy\right)\in\mathbb{Z}$, is defined. 
For the purposes of this article we utilize the following property of degree as our definition of the topological degree of a function $F$ at $\vy$ over a set $\Omega$. 
See O'Regan et al.~\cite{OrChCh2006} for details.
\begin{equation}\label{eq:Deg3}
  d\left(\Omega,F,\vy\right)=\sum\limits_{\vx\in F^{-1}(\vy)}\operatorname{sign}\left(J_F(\vx)\right)
\end{equation}
where $\operatorname{sign}\left(J_F(\vx)\right)$ denotes the sign of the determinant of the Jacobian of $F$ at $\vx$, i.e.,
\[\operatorname{sign}\left(J_F(\vx)\right)=   \left\{
\begin{array}{ll}
       \ -1   & \mbox{if } J_F(\vx)< 0, \\
      \quad 0 & \mbox{if } J_F(\vx)= 0,~\mbox{ and } \\
      \quad 1 & \mbox{if } J_F(\vx)> 0. \\
\end{array} 
\right. \]
We assume that the sum in \cref{eq:Deg3} evaluates to $0$ if $F^{-1}(\vy) = \emptyset$.
Additionally we utilize the following common properties of the topological degree. 
Again, see O'Regan et al.~\cite{OrChCh2006} for details. 

If $H : [0,1]\times\bar{\Omega}\rightarrow\R^n$ is continuous such that $H(t,\vx)\neq \vy~\forall t\in[0,1],~\vx\in\partial\Omega$,  then 
\begin{equation}\label{eq:Deg1} 
  d\left(\Omega,H(t,\cdot),\vy\right)\text{ does not depend on }t.
\end{equation}

\begin{equation}\label{eq:Deg2}
  \text{If } d(\Omega,F,\vy) \neq 0, \text{ then there exists } \vx \in \Omega \text{ such that } F(\vx)=\vy. 
\end{equation}

\subsection{New Theoretical Results}
In this section we will take full advantage of properties (\ref{eq:Deg3}), (\ref{eq:Deg1}), and (\ref{eq:Deg2}) as they apply to the Robust Feasibility Problem.
We begin by assuming there is a unique solution to the forecasted system at which point the Jacobian is non-zero.
We conclude by property (\ref{eq:Deg3}) that the degree is non-zero at $\vu^*$ for the forecasted system. 
We then utilize property (\ref{eq:Deg1}) to equate the degree of $\vu$ to the degree of $\vu^*$ for all $\vu$ satisfying the limits specified in Equation (\ref{eq:uLimits}) (under a proposed robustness margin), which by property (\ref{eq:Deg2}) allows us to guarantee solutions to the system under all realizations of $\vu$ satisfying the limits (in \ref{eq:uLimits}), i.e., verify the system is robust feasible for a given robustness margin.
Invoking property (\ref{eq:Deg1}), however, requires us to develop a homotopy that captures the system under all possible realizations of $\vu$ satisfying the limits in (\ref{eq:uLimits}).
Once we define such a homotopy we reduce the Robust Feasibility Problem to the problem of verifying the hypothesis of property (\ref{eq:Deg1}). 

\medskip
To that end let $F(\vx)=Q(\vx)+L\vx$, $\Omega=\{\vx| A\vx \leq \vb\}$ and $\hat{\vx}\in \Int(\Omega)$ be a solution to the forecasted system $F(\vx)=\vu^*$ given in \cref{eq:Quad}, such that $\operatorname{sign}\left(J_{F}(\hat{\vx})\right) \neq 0$.
For a review of efficient methods of verification that could be used here, see the work of Griewank \cite{GRIEWANK2014}. 
If no solution exists, then certainly the system is not robust feasible.
We define $\Omega_u=\{\vu \,|\,\vu \text{ satisfies limits in \cref{eq:uLimits}}\}$.
Our task is then to verify using existing methods or those we propose in this paper that no other solutions exist in $\Int(\Omega)$.
This step may require further restricting the domain or even a slight perturbation of the forecasted $\vu$. 
Thus by property (\ref{eq:Deg3}) we have verified that $d\left(\Omega, F(\vx), \vu^*\right)\neq 0$. 
Note that this is not the only method for verification, but in some sense is the easiest to carry out for our purposes. 

\medskip
We now introduce the homotopy we use to invoke property (\ref{eq:Deg1}).
Let  $\ell_{\vu^*}$ represent an arbitrary line passing through $\vu^*$ and let $\vl_{\min}$ and $\vl_{\max}$ be the two points of intersection of $\partial\Omega_u$ and $\ell_{\vu^*}$.
We define a homotopy $H_{\ell_{\vu^*}} : [0,1]\times\bar{\Omega}\rightarrow\R^n$ as 
\begin{align}
  H_{\ell_{\vu^*}}(t,\vx) = F(\vx)-\left[(1-t)\vl_{\min}+t\vl_{\max}\right]\,. \label{eq:Homo}
\end{align}
Based on this homotopy, we present the key result on verification of robust solvability problem.
%
\begin{lem}
  \label{lem:NaScondition}
  Let $\Omega=\{\vx| A\vx \leq \vb \}$, $\Omega_u=\{\vu \,|\,\vu\, \text{satisfies limits in \cref{eq:uLimits}}\}$ and $F(\vx)=Q(\vx)+L\vx$,  as described in Equations (\ref{eq:Quad}), (\ref{eq:xLimits}), and (\ref{eq:uLimits}). 
  If $d(\Omega,H_\ell\left(\frac{1}{2},\vx\right), \vzero) \neq 0$ for each choice of $\ell = \ell_{\vu^*}$, then the system is robust solvable if and only if the following statement holds:
  \begin{align}
    \not\exists \vx\in\partial\Omega, \vu\in\Omega_u \ \ \mbox{ such that } \ F(\vx)-\vu=\vzero \,. \label{eq:RSForm}
  \end{align}
\end{lem}

\begin{proof}
  Since $d\left(\Omega, F(\vx), \vu^*\right)\neq 0$ is verified by property (\ref{eq:Deg3}), to show the problem is robust solvable, it follows by the fact that there exists a unique solution $\hat{\vx}\in \Int(\Omega)$ to the forecasted system given in \cref{eq:Quad} such that  $H_{\ell_{\vu^*}}\left(\frac{1}{2},\hat{\vx}\right)=F(\hat{\vx})-\vu^*=\vzero$ and $ \operatorname{sign}\left(J_{H_{\ell_{\vu^*},\frac{1}{2}}}(\hat{\vx})\right) \neq 0$.
  And this condition holds if and only if $d(\Omega,H_\ell\left(\frac{1}{2},\vx\right),\vzero)\neq 0$ according to property (\ref{eq:Deg2}).
  Note that this property holds for all such lines passing through $\vu^*$ since for each $\hat{\vu}\in\Omega_u\setminus\{\vu^*\}$, there exists a line $\hat{\ell}_{\vu^*}$ passing through $\vu^*$ and $t\in[0,1]$, such that $\hat{\vu}=(1-t)\hat{\vl}_{\min}+t\hat{\vl}_{\max}$.
  Thus, when $d(\Omega,H_\ell\left(\frac{1}{2},\vx\right),\vzero)\neq 0$ for each choice of $\ell_{\vu^*}$,   we have 
  \[
    d(\Omega,H_\ell\left(\frac{1}{2},\vx\right),\vzero)\neq 0 \, \Longleftrightarrow \,
    F(\vx) - \vu \neq \vzero, \forall \vx \in \partial\Omega, \vu \in \Omega_u \,.
  \]
  Hence the system is robust solvable if and only if the statement  (\ref{eq:RSForm}) holds.
\end{proof}

Note that the statement (\ref{eq:RSForm}) is equivalent to property (\ref{eq:Deg1}) holding.
From here on we will assume $d(\Omega,H\left(\frac{1}{2},\vx\right),\vzero)\neq 0$ and focus our efforts on the development of methods for validating or invalidating the statement (\ref{eq:RSForm}).

\begin{lem} 
  \label{lem:BdOpt}
  Let $X\subset\R^n$ be full-dimensional and compact. If $\vF:\R^n\rightarrow\R^n$ is continuous, and 
  \[
  \min\limits_{\norm{\vlambda}=1} \, \max\limits_{\vx\in X} \, \vlambda^T\vF(\vx)
  \]
  obtains its optimal value at $\hat{\vx}$ and $\vlambda_{\hat{\vx}}$ then $\vF(\hat{\vx})\in \partial \vF(X)$. 
\end{lem}

\begin{proof} 
  We get the result by arriving at a contradiction.
  Assume $\vF(\hat{\vx})\in \vF(X)\setminus\partial \vF(X)$. 
  Let $\theta$ the angle between $\vlambda_{\hat{\vx}}$ and $F(\hat{\vx})$. 
  Thus
  \[
  \min\limits_{\norm{\vlambda}=1}\max\limits_{\vx\in X}\ \vlambda^T\vF(\vx) ~=~ \vlambda_{\hat{\vx}}^T\vF(\hat{\vx}) ~=~ \norm{\vF(\hat{\vx})}\cos(\theta).
  \]
  Since $X\subset\R^n$ is full-dimensional and compact, it has nonempty interior \cite{Ma1973}.
  And since $\vF$ is continuous, $\vF(X)$ is also nonempty and compact. 
  Hence there exists an $r>0$ such that  $B_r(\vF(\hat{\vx}))$, the ball of radius $r$ centered at $\vF(\hat{\vx})$, is in $\vF(X)\setminus\partial \vF(X)$. 
  Let $\vy$ be the antipodal point on $\partial B_r(\vF(\hat{\vx}))$ to the point of intersection between the line segment connecting the origin to $\vF(\hat{\vx})$ and $B_r(\vF(\hat{\vx}))$. 
  It follows then that $\norm{\vy}>\norm{\vF(\hat{\vx})}$ and $\theta$ is the angle between $\vlambda_{\hat{\vx}}$ and $\vy$.   
  Let $\vx^*\in X$ be such that $\vF(\vx^*)=\vy$.
  Such an $\vx^*$ exists since $\vy$ lies inside $\vF(x)$ which is nonepty and $\vF$ is continuous. 
  Therefore $\vlambda_{\hat{\vx}}^T\vF(\vx^*)=\norm{\vF(\vx^*)} \cos(\theta) > \norm{\vF(\hat{\vx})} \cos(\theta) = \vlambda_{\hat{\vx}}^T\vF(\hat{\vx})$, which is a contradiction. 
  The lemma now follows.
\end{proof}

We illustrate \cref{lem:BdOpt} in \cref{fig:illustration}.
The compact set $X\subset\R^2$ is shown in pink with boundary shown in blue.
For continuous $\vF:\R^2 \rightarrow \R^2$, let $\hat{\vx}$ be the point and $\vlambda_{\hat{\vx}}$ the unit vector that give the optimal value of $\min_{\norm{\vlambda}=1} \, \max_{\vx\in X} \, \vlambda^T\vF(\vx)$ as $\vlambda_{\hat{\vx}}^T\vF(\hat{\vx})$ based on \cref{lem:BdOpt}.
These vectors are shown in red.
In the case when the origin is contained in $\vF(X)$, as it happens here, we observe that there exists another point $\vx$ and unit vector $\vlambda_{\vx}$ satisfying $\vF(\vx) \in \partial \vF(X)$ (and thus $\vx\in \partial X$) such that we can get the lower bound of $\min_{\norm{\vlambda}=1}\max_{\vx\in X}\ \vlambda^T\vF(\vx)$ as $\min_{\vx\in \partial X} \max_{\norm{\vlambda}=1}\ \vlambda^T\vF(\vx)$ (shown in green).

\begin{figure}[htp!]
\begin{center}
  \includegraphics[scale=0.32]{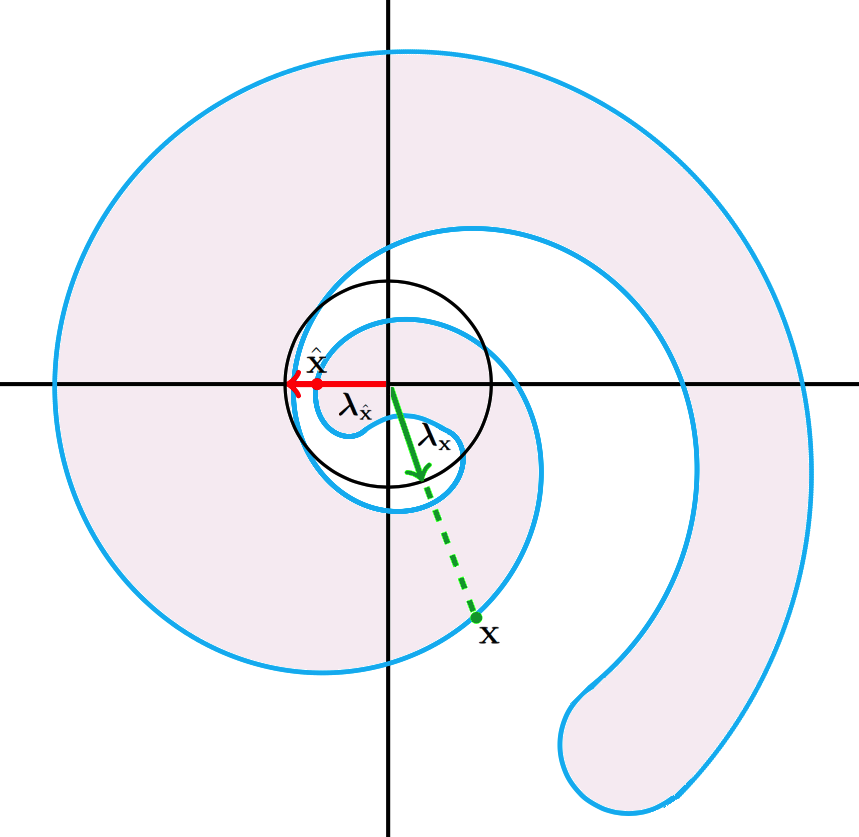}
\end{center}
\caption{\label{fig:illustration}
  Illustration of \cref{lem:BdOpt}.
  The unit vectors $\vlambda_{\hat{\vx}}$ and $\vlambda_{\vx}$ are shown as the red and green arrows, respectively.
}
\end{figure}

We formalize the last observation for the general case in our main theorem, which characterizes the structure of the function $\vF$.

\begin{thm} 
  \label{thm:MainIneq}
  Let $X\subset\R^n$ be full-dimensional and compact, $\vF:\R^n\rightarrow\R^n$ be continuous such that $\vF(X)$ contains the origin. 
  If $\vF$ is injective over $X$ then
  \[
  \min\limits_{\norm{\vlambda}=1}\max\limits_{\vx\in X}\ \vlambda^T\vF(\vx) ~\geq~ \min\limits_{\vx\in \partial X} \max\limits_{\norm{\vlambda}=1}\ \vlambda^T\vF(\vx).
  \]
  %
  \begin{proof}
    Consider the origin lies on the boundary of $\vF(X)$, and let $\hat{\vx} \in \partial X$ such that $F(\hat{\vx})=\vzero$. 
    Such an $\hat{\vx}$ exists since $\partial F(X) = F(\partial X)$ by the Invariance of Domain theorem, as $F$ is continuous and injective over a compact set (so, an interior point cannot get mapped to a boundary point). 
    Then  clearly
    \[
       \min\limits_{\norm{\vlambda}=1}\max\limits_{\vx\in X}\ \vlambda^T\vF(\vx) ~\geq~
       \min\limits_{\norm{\vlambda}=1}\vlambda^T\vF(\hat{\vx}) ~=~ 0 ~=~
       \max\limits_{\norm{\vlambda}=1}\ \vlambda^T\vF(\hat{\vx}) ~\geq~
       \min\limits_{\vx\in \partial X}\max\limits_{\norm{\vlambda}=1}\ \vlambda^T\vF(\vx).
    \]
    Hence assume the origin lies in the interior of $\vF(X)$.
    Let $\hat{\vx}$ be the point and $\vlambda_{\hat{\vx}}$ the unit vector at which
    \[
    \min\limits_{\norm{\vlambda}=1}\max\limits_{\vx\in X}\ \vlambda^T\vF(\vx)
    \]
    obtains its optimal value.
    We consider two cases.

    Case 1: If the angle $\theta$ between $F(\hat{\vx})$ and $\vlambda_{\hat{\vx}}$ is $0$, then $\displaystyle \max\limits_{\norm{\vlambda}=1}\vlambda^T\vF(\hat{\vx}) = \vlambda_{\hat{\vx}}^T\vF(\hat{\vx})$. 
    Furthermore by \cref{lem:BdOpt}, $\vF(\hat{\vx}) \in \partial \vF(X)$ and thus $\hat{\vx}\in \partial X$ since $\vF$ is injective, i.e., $\vF$ maps $\partial X$ to $\partial \vF(X)$, by hypothesis. 
    It follows now that
    \[
      \min\limits_{\vx\in \partial X}\max\limits_{\norm{\vlambda}=1}\ \vlambda^T\vF(\vx) ~\leq~
      \max\limits_{\norm{\vlambda}=1}\ \vlambda^T\vF(\hat{\vx}) ~=~
      \vlambda_{\hat{\vx}}^T\vF(\hat{\vx}) ~=~
      \min\limits_{\norm{\vlambda}=1}\max\limits_{\vx\in X}\ \vlambda^T\vF(\vx).
    \]

    Case 2: Assume the angle $\theta \neq 0$ (between $F(\hat{\vx})$ and $\vlambda_{\hat{\vx}}$). 
    Let $\vx^*$ be a point on the boundary of $X$ such that the angle between $\vlambda_{\hat{\vx}}$ and $\vF(\vx^*)$ is 0. 
    Such a point must exist as $\vF(X)$ is compact and by hypothesis $\vF(X)$ contains the origin, is injective,
    and by assumption the origin lies in the interior of $\vF(X)$. 
    It follows then that
    \[
      \min\limits_{\vx\in \partial X}\max\limits_{\norm{\vlambda}=1}\ \vlambda^T\vF(\vx) ~\leq~
      \max\limits_{\norm{\vlambda}=1}\vlambda^T \vF(\vx^*) ~=
      \vlambda_{\hat{\vx}}^T\vF(\vx^*) ~\leq
      \max\limits_{\vx\in X}\vlambda_{\hat{\vx}}^T\vF(\vx) ~=~
      \min\limits_{\norm{\vlambda}=1}\max\limits_{\vx\in X}\ \vlambda^T\vF(\vx).
    \]
    The theorem now follows.
  \end{proof}
\end{thm}

\cref{thm:MainIneq} and \cref{lem:BdOpt} provide us with the theoretical tools we need to develop procedures for approximating the robustness margin.
The hypothesis of \cref{thm:MainIneq} does however require us to assume the system is injective under the constraints in Equation (\ref{eq:xLimits}).
However, injectivity is only required to ensure $\partial F(X) = F(\partial X)$, and thus we can generalize to systems that are not necessarily injective if they yet retain $\partial F(X) = F(\partial X)$ as an applicable property. 
With this in mind we carry with us the necessary property $\partial F(X) = F(\partial X)$ throughout the rest of the article. 

We will use the terminology \enquote{inner bound procedures} to describe the processes of verifying robust feasibility while expanding the uncertainty box centered at $\vu^*$ in order to compute the lower bounds on the robustness margin, which these procedures undertake. 
We use the terminology \enquote{outer bound procedures} to capture in a similar fashion the procedures used to compute the upper bounds on the robustness margin by contracting the uncertainty box until the system may be robust feasible. 
As such we dedicate the next two sections to the development of these inner and outer bound formulations. 

\section{Computing Lower Bounds on the Robustness Margin} \label{sec:inbdform}

In this section we will derive procedures for computing a lower bound on the robustness margin. 
We start with an exact formulation, which turns out to be hard to implement efficiently in practice.
Hence we relax the procedures until they become computationally tractable. 
We end the section by providing three different implementations of our final derived, relaxed, computationally tractable procedure specified in Equation (\ref{eq:OPTfeasrelax}).
Each of these three practical implementations brings a unique set of attributes, which makes none of them the clearly preferred candidate.

\smallskip
\begin{thm}
  Let $\Omega=\{\vx| A\vx \leq \vb\}$, $\Omega_{u}=\{\vu| u^{\min}_i\leq u_i \leq u^{\max}_i \ \forall i \}$, and $F(\vx)=Q(\vx)+L\vx$ as described in Equations (\ref{eq:Quad}), (\ref{eq:xLimits}), and (\ref{eq:uLimits}). 
  Let
  \[
  z = \min\limits_{\vx \in \partial \Omega, \vu \in \Omega_{u} }\max\limits_{\norm{\vlambda}=1}\ \vlambda^T\left(F(\vx)-\vu\right).
  \]
  If there is an $r>0$ such that $r \leq e_i \ \forall i$ with $\ e_i>0$, where $e_i$ denotes the error bounds associated with $ u^{\min}_i$ and $ u^{\max}_i$, and if $z > 0$ then the system is robust feasible and has a robustness margin of at least $r$.

  \medskip
  \begin{proof} 
    If \cref{eq:RSForm} is invalidated then there exists $\hat{\vx} \in \partial\Omega$ such that $F(\hat{\vx})=\hat{\vu}$ for some $\hat{\vu}\in\Omega_{u}$ and thus
    \[
    z = \min\limits_{\vx \in \partial \Omega, \vu \in \Omega_{u}}\max\limits_{\norm{\vlambda}=1}\ \vlambda^T\left(F(\vx)-\vu\right) ~\leq~ \max\limits_{\norm{\vlambda}=1, \vx=\hat{\vx}}\ \vlambda^T\left(F(\vx)-\hat{\vu}\right) ~=~ 0.
    \]
    Hence if $z>0$, \cref{eq:RSForm} is validated, and the system is robust feasible.
    It follows then by definition that the system has a robustness margin of at least $r$.
  \end{proof}
\end{thm}

  \medskip
\begin{thm} \label{thm:RobFeas}
  Let $\partial\Omega_i=\{\vx| (A\vx)_i = b_i, A\vx\leq \vb\}$, $\Omega_{u}=\{\vu| u^{\min}_i\leq u_i \leq u^{\max}_i \ \forall i \}$, and $F(\vx)=Q(\vx)+L\vx$ as described in \cref{eq:Quad,eq:xLimits,eq:uLimits}.
  Define
  \begin{align}
    \rz_i =  \min_{\vx\in\partial\Omega_i, \vu \in \Omega_u} \norm{F(\vx)-\vu}. \label{eq:OPTfeas}
  \end{align}
  The system is robust feasible if and only if $\rz_i>0$ for each $i = 1, \ldots, m$, where $m$ is the number of rows of $A$.

  \begin{proof} \ \\
    $\boxed{\Rightarrow}$ \\ 
    If the system is not robust feasible then there exists $\hat{\vu}'\in \Omega_u$ such that $F(\vx)=\hat{\vu}'$ has no interior solution. 
    Since $F(\vx)=\hat{\vu}'$ has a solution and $F$ is continuous over a compact domain, there must exist an $\hat{\vx} \in \Omega = \{\vx| A\vx \leq \vb\}$ such that $F(\hat{\vx})=\hat{\vu}$ for some $\hat{\vu}\in \Omega_u\cap \partial F(X)$. 
    Since $\partial F(X) = F(\partial X)$ we have that $\hat{\vx}\in\partial\Omega$, which implies that there is an $i$ such that $A\hat{\vx}_i=b_i$, and thus $0 \leq \rz_i \leq \norm{F(\hat{\vx})-\hat{\vu}}=0$.\\
    $\boxed{\Leftarrow}$ \\ 
    If there exists a $i$ such that $z_i=0$, then it follows that there must be the $\hat{\vx}\in \partial \Omega$ which contains $\hat{x_i}$ as an element such that $F(\hat{\vx})=\hat{\vu}$ for some $\hat{\vu}\in \Omega_u$.
    Thus there is no \emph{interior} point $\vx$ to make the equation $F(\vx)=\hat{\vu}$ hold.
    Then based on \cref{RobustDef}, the system is not robust feasible. 
  \end{proof}
\end{thm}


The optimization problems presented in \cref{thm:RobFeas} are nonlinear and nonconvex.
Hence it may be difficult to solve them in general.
In fact, since $F(\vx)$ is quadratic in $\vx$, they are quadratically constrained quadratic programs (QCQPs), which are NP-hard in general \cite{PaBo2017}.
At the same time, we can use well known \emph{semidefinite programming relaxations for QCQPs} to obtain lower bounds on the optimal values \cite{VaBo1996}.
Since $Q$ is quadratic, it can also be written as a linear function of $\vx\vx^T$.
More concretely, we can write
$$Q(\vx)_i=\vx^TQ_i\vx=\Tr(Q_i\vx\vx^T)$$
where each $Q_i$ is as defined in \cref{eq:Quad}. 
Since $\vx$ should satisfy $A\vx\leq \vb$, we get
\[
  (\vb-A\vx)(\vb-A\vx)^T \geq 0~~\Rightarrow~~ \vb\vb^T-A\vx\vb^T-\vb(A\vx)^T+A(\vx\vx^T)A^T\geq 0.
\]
If we allow a symmetric positive semidefinite matrix $X$ to take the place of $\vx\vx^T$ and drop the rank constraint ($\rank(X)=1$)) we can construct the following relaxation for the optimization problem presented in \cref{thm:RobFeas}:
\begin{equation}\label{eq:OPTfeasrelax}
  \begin{array}{rl}
    \hat{\rz}_{i} = & \min\limits_{\vx}  b_i-(A\vx)_i  \\
    \vspace*{-0.05in} \\
    \text{subject to } \ & \Tr\left(Q_iX\right)+ L_i\vx \geq \vu_i^{\min} \ \ \forall i\\
    & \Tr\left(Q_iX\right)+ L_i\vx \leq \vu_i^{\max} \ \ \forall i\\
    &A\vx\leq \vb \\
    &\vb\vb^T-A\vx\vb^T-\vb(A\vx)^T+AXA^T\geq O \\
    &X \text{ is symmetric and positive semidefinite.}
  \end{array}
\end{equation}
%
Here $L_i$ denotes the $i^{\rm th}$ rows of $L$, and $O$ denotes the $2n \times 2n$ matrix of zeros with the constraints understood to be component-wise inequalities.
Note that if we impose $\rank(X)=1$, we do get an exact formulation.
At the same time, it may be difficult to impose this constraint.
But the system (in \cref{eq:OPTfeasrelax}) without the rank constraint is a convex optimization problem (a semidefinite program, in fact) and can be solved efficiently. 
Since this is a relaxation of the procedure described in \cref{thm:RobFeas}, if $\hat{\rz}_i>0$ for each $i$, the condition in \cref{thm:RobFeas} (i.e., $\rz_i>0 \ \forall i$) is satisfied. 
We can further relax the formulation in \cref{eq:OPTfeasrelax} by dropping the condition that $X$ be positive semidefinite, which transforms the problem from a semidefinite program to a linear program. 
We now present three formulations for this new, relaxed program, and provide numerical results for each of them in \cref{ssec:compres}.
We also describe the advantages and drawbacks of each formulation.

\bigskip
\bigskip
\textbf{LP Feasibility Procedure} 
\begin{equation} \label{eq:OPTfeasrelaxLP1}
\begin{array}{rl}
 &\text{Find an }\vx \\
 \text{subject to } \ &(A\vx)_i= b_i \\
 &\Tr\left(Q_iX\right)+ L_i\vx \geq \vu_i^{\min}  \ \ \forall i\\
 & \Tr\left(Q_iX\right)+ L_i\vx \leq \vu_i^{\max}  \ \ \forall i\\
 	&A\vx\leq \vb \\
 	&\vb\vb^T-A\vx\vb^T-\vb(A\vx)^T+AXA^T\geq O \\
 	&X \text{ is symmetric.}
\end{array}
\end{equation}

This procedure has the advantage of being a linear program and hence can be solved efficiently.
But as noted by the objective $\rz_i$, one must iterate over each dimension of $A\vx$ checking feasibility of the procedure. 
Of course, should the procedure prove feasible then we have found a solution on the boundary and thus invalidating the Statement in (\cref{eq:RSForm}). 
If the procedure proves infeasible then we are free to push the robustness margin higher and test again. 
An alternative approach is to consider all of the dimensions of $A\vx$ simultaneously by introducing extra binary variables and creating a MIP as follows. 

\bigskip
\textbf{MIP Procedure}\
\begin{equation}\label{eq:OPTfeasrelaxLP2}
\begin{array}{rl}
\max &  z  \\
 \text{subject to } \ &\Tr\left(Q_iX\right)+ L_i\vx \geq \vu_i^{\min}  \ \ \forall i\\
 & \Tr\left(Q_iX\right)+ L_i\vx \leq \vu_i^{\max}  \ \ \forall i\\
 	&A\vx\leq \vb \\
 	&\vb\vb^T-A\vx\vb^T-\vb(A\vx)^T+AXA^T\geq O \\
 	&X \text{ is symmetric} \\
 	& z\leq (A\vx)_i-b_i+R(1-d_i) \ \forall i \ \text{ for some large enough $R$} \\
 	& \sum\limits_i d_i=1 \\
 	& d_i\in\{0,1\} \ \forall i .
\end{array}
\end{equation}

The MIP and LP Feasibility procedures are similar and should theoretically give the same results. 
However, as we will show in our numerical studies, the LP Feasibility procedure could outperform the MIP procedure by producing higher robustness margins and running faster in higher dimensions. 
In both procedures, the process ends after a boundary solution is found. 
This solution may not be a boundary solution to the actual system, but may be an artifact of the relaxations used to create the procedures. 
One way of tackling this issue is by updating the constraints using an iterative process as outlined in the following procedure.

\bigskip
\textbf{LP Bound Tightening Procedure}
\begin{equation}\label{eq:OPTfeasrelaxLP3}
\begin{array}{rl}
\max &  z_i = (A\vx)_i  \\
 \text{subject to } \ &\Tr\left(Q_iX\right)+ L_i\vx \geq \vu_i^{\min}  \ \ \forall i\\
 & \Tr\left(Q_iX\right)+ L_i\vx \leq \vu_i^{\max}  \ \ \forall i\\
 	&A\vx\leq \vb \\
 	&\vb\vb^T-A\vx\vb^T-\vb(A\vx)^T+AXA^T\geq O \\
 	&X \text{ is symmetric.}
\end{array}
\end{equation}

What distinguishes the LP Bound Tightening procedure from the other two procedures is the ability to use it iteratively by updating the constraints of \cref{eq:OPTfeasrelaxLP3}, replacing $\vb$ with $\vz$, the vector of $z_i$'s found after running the procedure over all dimensions of $A\vx$. 
Since clearly $\vz \leq \vb$, we have that the polytope $\{\vx ~|~ A\vx\leq \vz\} \, \subseteq \, \{\vx ~|~ A\vx\leq \vb\}$. 
Thus it follows that any system deemed robust feasible using \cref{eq:OPTfeasrelaxLP1} or \cref{eq:OPTfeasrelaxLP2} will certainly be found robust feasible using \cref{eq:OPTfeasrelaxLP3}, but a system found robust feasible using \cref{eq:OPTfeasrelaxLP3} may not be found robust feasible using \cref{eq:OPTfeasrelaxLP1} or \cref{eq:OPTfeasrelaxLP2}. 
The drawback of the bound tightening procedure, as we shall see in the Computational Results (\cref{ssec:compres}), is choice of parameters to be manually set in order to tell the procedure when to stop.

\section{Computing Upper Bounds on the Robustness Margin} \label{sec:outbdform}  

Due to the hardness of solving the problem $~\displaystyle \min\limits_{\vx\in \partial \Omega, \vu\in\Omega_u}\max\limits_{\norm{\vlambda}=1 }\ \vlambda^T\left(F(\vx)-\vu\right),\,$ we derived a sequence of relaxations in order to arrive at a method that was computationally efficient for non-trivial data sets. 
We can similarly work with $~\displaystyle \min\limits_{\norm{\vlambda}=1}\max\limits_{\vx\in \bar{\Omega}, \vu\in\Omega_u}\ \vlambda^T\left(F(\vx)-\vu\right) \,$ to produce computationally efficient procedures to compute upper bounds for the robustness margin. 

\bigskip
\begin{thm}\label{thm:OPTfeasOut} 
Let $\bar{\Omega}=\{\vx| A\vx\leq \vb\}$, $\Omega_{u}=\{\vu| u^{\min}_i\leq u_i \leq u^{\max}_i \ \forall i \}$, and $F(\vx)=Q(\vx)+L\vx$ as described in \cref{eq:Quad,eq:xLimits,eq:uLimits}. 
Let
\[
  z = \min\limits_{\norm{\vlambda}=1}\max\limits_{\vx\in \bar{\Omega}, \vu\in\Omega_u}\ \vlambda^T\left(F(\vx)-\vu\right).
\]
If there is an $r>0$ such that $r \leq e_i \ \mbox{ for all } i \mbox{ with } e_i>0$, where $e_i$ is the error bound associated with $ u^{\min}_i$ and $ u^{\max}_i$, and if $z=0$ then the system has robustness margin of no more than r.

\medskip
\begin{proof} 
  Observe by \cref{lem:BdOpt}, if $z_i = \min_{\|\lam\|=1}\max_{\vx\in\bar{\Omega}, \vu\in\Omega_u}\lam^T(F(\vx)-\vu^*)=0$, then we have $F(\hat{\vx})\in \partial F(X)=F(\partial X)$. Thus, it's followed by the statement that there exist a $\vx\in \partial X$ such that $F(\vx)-\vu=0$ for some $u\in \Omega_u$ which invalids \cref{eq:RSForm}.
%
\end{proof}
\end{thm}

\bigskip
We can relax the computation suggested in \cref{thm:OPTfeasOut} by utilizing the same techniques as before, replacing $\vx\vx^T$ with a positive semidefinite matrix $X$, with the option of dropping the condition that $X$ be positive semidefinite.
This step transforms the computation from a semidefinite program to a linear or mixed integer program depending on how one deals with the constraint $\norm{\vlambda}=1$. 
In our tests we use the $\ell_1$ norm and introduce variables that capture the absolute value of each $\lambda_i$.
We obtain the following procedure.

\bigskip
\bigskip
\textit{Outer Bound Procedure}: Formulation
\begin{equation}\label{eq:OPTfeasOutRelaxa}
\begin{array}{rl}
 \ z &=\min\limits_{\norm{\vlambda}=1}\max\limits_{\vx} \, \vlambda^T\left(F(\vx)-\vu^* \right) \\
 \text{subject to } \ & A\vx\leq \vb \\
 	&\vb\vb^T-A\vx\vb^T-\vb(A\vx)^T+AXA^T\geq O \\
 	&X \text{ is symmetric.}
\end{array}
\end{equation}

\bigskip




In order to implement this procedure, we use linear programming duality to write the problem in \cref{eq:OPTfeasOutRelaxa} as direct minimization LP (in place of a min-max problem).
To construct the dual of the inner maximum objective function, we first write the constraints as $M\hat{\vx} \leq \vB$ and the objective as $g(Q,L,\lambda)\hat{\vx}-\lambda^T\vu^*$, where $\hat{\vx}^T=[\vx^T ~\hat{X}^T]$, with $\hat{X}$ being the vector form of the upper triangular (including diagonal) portion of $X$. 
We enforce the constraint that $X$ is symmetric by utilizing only the upper triangular portion of $X$.
We can write $M \hat{\vx} \leq \vB$ using only the upper triangular entries including the diagonal as follows.

\vspace*{0.3in}
\noindent
\resizebox{\hsize}{!}{%
$
  \begin{bmatrix}
    A_{11} & \dots & A_{1n} & 0 & \dots & 0 \\ 
    \vdots & \ddots & \vdots & \vdots & \ddots & \vdots\\
    A_{m1} & \dots & A_{mn} & 0 & \dots & 0 \\
    \vb_1A_{11}+\vb_1A_{11} & \dots & \vb_1A_{1n}+\vb_1A_{1n} & -A_{11}A_{11} & \dots & -A_{1n}A_{1n} \\ 
    \vdots & \ddots & \vdots & \vdots & \ddots & \vdots\\
    \vb_rA_{q1}+\vb_qA_{r1} & \dots & \vb_rA_{qn}+\vb_qA_{rn} & -A_{q1}A_{r1} & \dots & -A_{1n}A_{2n} \\ 
    \vdots & \ddots & \vdots & \vdots & \ddots & \vdots\\
    \vb_mA_{m1}+\vb_mA_{m1} & \dots & \vb_mA_{mn}+\vb_mA_{mn} & -A_{m1}A_{m1} & \dots & -A_{mn}A_{mn} 
  \end{bmatrix}
  \begin{bmatrix}
    x_1 \\ 
    \vdots \\
    x_n \\
    X_{11} \\ 
    \vdots \\
    X_{qr}\\
    \vdots \\
    X_{nn} 
  \end{bmatrix}%
  ~\leq~
  \begin{bmatrix}
    b_1 \\ 
    \vdots \\
    b_m \\
    b_1b_1 \\ 
    \vdots \\
    b_qb_r\\
    \vdots \\
    b_mb_m 
  \end{bmatrix}.%
$
}%
%

\vspace*{0.3in}
We can write the objective function as $g(Q,L,\lambda)\hat{\vx}-\lambda^T\vu^*$ in the following way.
%
\[
  \begin{bmatrix}
    \sum\limits_{j=1}^nL_{j,1}\lambda_j & 
    \dots &
    \sum\limits_{j=1}^nL_{j,n}\lambda_j &
    \sum\limits_{j=1}^n\lambda_jQ_{11}^j & 
    \dots &
    \sum\limits_{j=1}^n\lambda_jQ_{nn}^j 
  \end{bmatrix}
  \begin{bmatrix}
    x_1 \\ 
    \vdots \\
    x_n \\
    X_{11} \\ 
    \vdots \\
    X_{nn} 
  \end{bmatrix}%
  -
  \begin{bmatrix}
    \lambda_1 &  \dots & \lambda_n 
  \end{bmatrix}
  \begin{bmatrix}
    u^*_1 \\ 
    \vdots \\
    u^*_n \\
  \end{bmatrix}.%
\]
If $\vu^*=\vzero$, then the objective function reduces to just $g(Q,L,\lambda)\hat{\vx}$.

\vspace*{0.3in}
On the other hand, if $\vu^*\neq \vzero$ then we can add an extra dummy variable $\vx_{n+1}$ with the constraint $\vx_{n+1}=1$, to obtain the following system.

\vspace*{0.3in}
\noindent
\resizebox{\hsize}{!}{%
  $
  \begin{bmatrix}
    A_{11} & \dots & A_{1n} & A_{1,n+1} & 0 & \dots & 0 \\ 
    \vdots & \ddots & \vdots & \p \vdots & \vdots & \ddots & \vdots\\
    A_{m1} & \dots & A_{mn} & A_{mn+1} & 0 & \dots & 0 \\
    0 & \dots & 0 & \p 1 & 0 & \dots & 0 \\
    0 & \dots & 0 & -1 & 0 & \dots & 0 \\
    \vb_1A_{11}+\vb_1A_{11} & \dots & \vb_1A_{1n}+\vb_1A_{1n} & \p 0 & -A_{11}A_{11} & \dots & -A_{1n}A_{1n} \\ 
    \vdots & \ddots & \vdots & \p \vdots & \vdots & \ddots & \vdots\\
    \vb_rA_{q1}+\vb_qA_{r1} & \dots & \vb_rA_{qn}+\vb_qA_{rn} & \p 0 & -A_{q1}A_{r1} & \dots & -A_{1n}A_{2n} \\ 
    \vdots & \ddots & \vdots & \p \vdots & \vdots & \ddots & \vdots\\
    \vb_mA_{m1}+\vb_mA_{m1} & \dots & \vb_mA_{mn}+\vb_mA_{mn} & \p 0 & -A_{m1}A_{m1} & \dots & -A_{mn}A_{mn} 
  \end{bmatrix}%
  \begin{bmatrix}
    x_1 \\ 
    \vdots \\
    x_n \\
    x_{n+1} \\
    X_{11} \\ 
    \vdots \\
    X_{qr}\\
    \vdots \\
    X_{nn} 
  \end{bmatrix}%
  ~\leq~
  \begin{bmatrix}
    \p b_1 \\ 
    \p \vdots \\
    \p b_m \\
    \p 1 \\
    -1 \\
    \p b_1b_1 \\ 
    \p \vdots \\
    b_qb_r\\
    \p \vdots \\
    b_mb_m 
  \end{bmatrix}.%
  $
}
%
We can then write the objective function $g(Q,L,\lambda)\hat{\vx}$ as follows.
\[
  \begin{bmatrix}
    \sum\limits_{j=1}^nL_{j,1}\lambda_j &
    \dots &
    \sum\limits_{j=1}^nL_{j,n}\lambda_j &
    \sum\limits_{j=1}^n-\lambda_j u^*_j &
    \sum\limits_{j=1}^n\lambda_jQ_{11}^j &
    \dots &
    \sum\limits_{j=1}^n\lambda_jQ_{nn}^j 
  \end{bmatrix}
  \begin{bmatrix}
	x_1 \\ 
	\vdots \\
	x_n \\
	x_{n+1} \\
	X_{11} \\ 
	\vdots \\
	\vdots \\
	X_{nn} 
  \end{bmatrix}.%
\] 

In either case, the final optimization problem, and hence the procedure we use in our tests, is given as follows.

\medskip
\textbf{Outer Bound Procedure} 
\begin{equation}\label{eq:OPTfeasOutRelaxb}
\begin{array}{rl}
  \ z &=\min\limits_{\norm{\vlambda}=1, \vy}\vB^T\vy  \\
  \vspace*{-0.15in} \\
 \text{subject to } \ & M^T\vy = g(Q,L,\lambda)^T \\
 & \vy \geq \vzero .
\end{array}
\end{equation}


\section{Illustrative Example} \label{sec:expl}

We present a toy example to illustrate the lower bound feasibility and bound tightening procedures as well as the upper bound approximation method.
We take the problem data as follows.

\medskip
\noindent
\resizebox{\hsize}{!}{%
$
A=
\begin{bmatrix}
  -1 & \p 0  \\ 
\p 1 & \p 0  \\  
\p 0 & -1 \\
\p 0 & \p 1  
\end{bmatrix},%
~
\vb =
\begin{bmatrix}
  -0.5  \\ 
\p 3 \\  
  -0.5 \\
\p 3  
\end{bmatrix}, %
~
Q_1=
\begin{bmatrix}
  1 & 0  \\ 
  0 & 0  
\end{bmatrix},%
~
Q_2=
\begin{bmatrix}
  0 & 0  \\ 
  0 & 1  
\end{bmatrix},%
~
L=
\begin{bmatrix}
  1 & -3  \\ 
  2 & -1  
\end{bmatrix}, %
~\mbox{ and }
\vu^*=
\begin{bmatrix}
  -2  \\ 
\p 4   
\end{bmatrix}.%
$
}
%

\medskip
\noindent These components of data correspond to the following quadratic system.
\[
F(\vx) =
\begin{bmatrix}
  x_1^2 +   x_1 - 3x_2 \\
  x_2^2 +  2x_1 - x_2
\end{bmatrix}
=
\begin{bmatrix}
    -2 \\
    \p 4
\end{bmatrix}
\hspace*{0.15in}
\text{ with }
\hspace*{0.15in}
\Omega_{\vx} = \left\{ \vx ~\big\vert~
\begin{bmatrix}
  0.5 \\
  0.5
\end{bmatrix}
\leq 
\begin{bmatrix}
  x_1 \\
  x_2
\end{bmatrix}
\leq
\begin{bmatrix}
  3 \\
  3
\end{bmatrix}
\right\}.
\]
This system has a unique solution in $\Omega_{\vx}$ given by
\[
\vx =
\begin{bmatrix}
x_1 \\
x_2
\end{bmatrix}
\approxeq
\begin{bmatrix}
1.36 \\
1.74
\end{bmatrix}.
\] 
Running the Outer Bound Procedure (\cref{eq:OPTfeasOutRelaxb}) gives an upper bound on the robustness margin of $2.63462$.
Running the LP feasibility procedure results in an lower bound on the robustness margin of $1.20454$, while the bound tightening yields a lower bound of $1.706649$. 
We illustrate $F(\Omega_{\vx})$ with $\vu^*$ in \cref{fig:FOmega}.
Notice that the bound tightening procedure produces a better lower bound approximation as theoretically predicted.

\begin{figure}[htp!]
  \begin{center}
    \includegraphics[scale=0.5425]{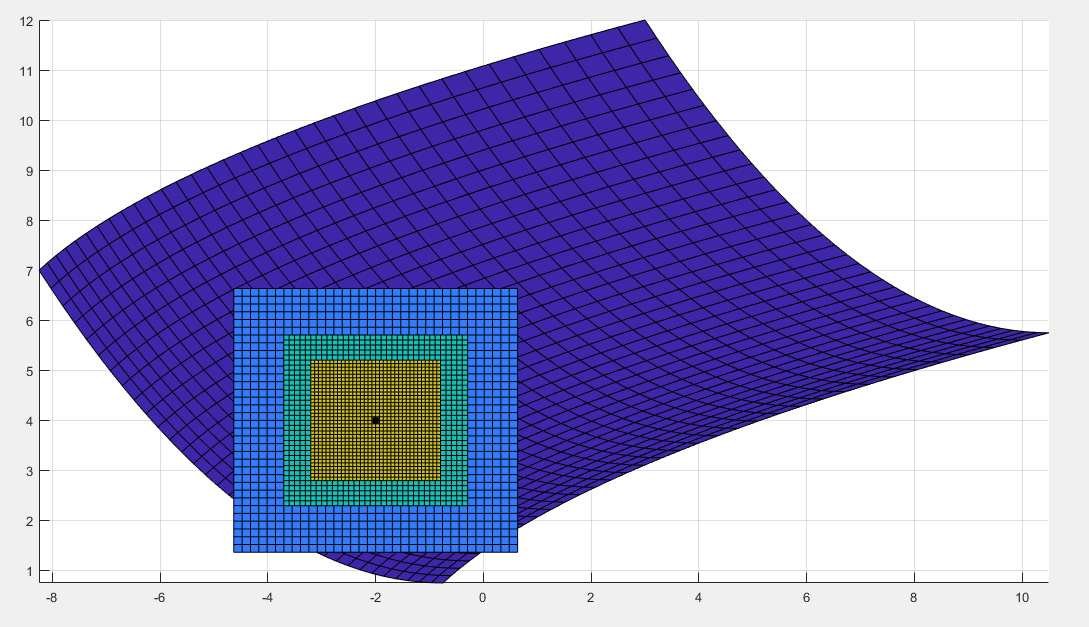} 
  \end{center}
  \caption{\label{fig:FOmega}
    Illustration of $F(\Omega_x)$ (dark blue surface), $\vu^*$ (center of the boxes), and $\Omega_{u^*}$  with radii given by the upper bound procedure (light blue box), bound tightening procedure (turquoise box), and feasibility procedure (yellow box).
  }
\end{figure}

\vspace*{-0.15in}
\section{Implementation on Power Systems} \label{sec:numstd}  
\vspace*{-0.05in}

We now present results from computational studies that demonstrate the efficiency of the bounding procedures we have introduced.
For all the numerical results presented, we apply the LP and MIP feasibility procedures (\cref{eq:OPTfeasrelaxLP1,eq:OPTfeasrelaxLP2}), the LP bound tightening procedure (\cref{eq:OPTfeasrelaxLP3}), and the Outer bound procedure (\cref{eq:OPTfeasOutRelaxb}) to find lower and upper bounds for the robustness margins with respect to the optimal power flow equations derived using datasets obtained from the MatPower package found in the MATLAB software \cite{matpower} and the NICTA Energy System Test Case Archive (NESTA) \cite{CoGoSc2019}.
We specifically show results for tests conducted on cases 5, 9, 14, 29, 30, 39 and 57. 
In every case the power flow equations were converted into the form of a system of quadratic equations we study (as described in \cref{eq:Quad,eq:xLimits,eq:uLimits}). 
To simulate real life scenarios we allowed the first 5 dimensions of $\vu$ to represent renewable energy, and the others representing no variation. We then slowly increased the variation of the first 5 dimensions of $\vu$, while utilizing procedures for the upper and lower bound verification to verify robust feasibility. 
We first detail specifically how this transformation was conducted.

\vspace*{-0.10in}
\subsection{OPF to Quadratic System} \label{ssec:opf2qsys}

As described by Dvijotham et al.~\cite{DjTuritsyn}, the AC power flow equations can be written as follows.
\begin{equation}\label{eq:Real1}
	\begin{array}{rl}
	\Re\left(\sum\limits_{k=1}^n V_i\left(\overline{Y_{ik}V_k} + \overline{Y_{i0}V_0}\right)\right) &= p_i, \ \forall i\in PQ \\
	
	\Im\left(\sum\limits_{k=1}^n V_i\left(\overline{Y_{ik}V_k} + \overline{Y_{i0}V_0}\right)\right) &= q_i, \ \forall i\in PQ \\
	
	\Re\left(\sum\limits_{k=1}^n V_i\left(\overline{Y_{ik}V_k} + \overline{Y_{i0}V_0}\right)\right) &= p_i, \ \forall i\in PV \\ 
	
	|V_i|^2 &= v_i, \  \forall i \in PV,
	\end{array}
\end{equation}
%
where $V_i$ denotes the complex voltage phasor, $p_i$ the active and $q_i$ the reactive power injection, and $Y$ the admittance matrix at node $i$.
$PV$ denotes the set of PV or \emph{Generator} nodes/buses, $PQ$ denotes the set of PQ or \emph{Load} buses, and $v_i$ denotes the squared voltage magnitude setpoints at the $PV$ buses.
We can then rewrite \cref{eq:Real1} into the system outlined in \cref{eq:Quad,eq:xLimits,eq:uLimits} by setting
\[
\begin{array}{rl}
\vx= & \begin{bmatrix} \Re(V_1) & \dots & \Re(V_n) & \Im(V_1) & \dots  &  \Im(V_n) \end{bmatrix}^T ~~\mbox{and} \\
\vu=& \begin{bmatrix} p_1 &  \dots &  p_n &  q_1 &  \dots &  q_n &  v_1 &  \dots &  v_n \end{bmatrix}^T.
\end{array}
\]

\subsection{Computational Results} \label{ssec:compres}

The procedures were computed with $A \vx \leq \vb = B \vone$ for $B \in\{0.001, 0.005,0.01\}$.
$A$ in these cases is a matrix such that $A \vx \leq \vb$ controls the flow between nodes in the power grid, i.e., each row of $A\vx \leq \vb$ has the form $x_i-x_j\leq b_k$. 
All computations were performed on a laptop running the  64bit MacOS Catalina operating system containing an 2.3GHz dual-core Intel Core i5, Turbo Boost up to 3.6GHz, with 64MB of eDRAM.
Details on the computation are given in \cref{tab:exp}, for determining the practical scaling properties of these procedures with the fixed $B=0.001$.
We display the data on a case by case basis to highlight the effect of allowing more fluctuation between the nodes, i.e., as $B$ increases, in Figure \ref{fig:Graphs1}. 
 
\begin{table}[!htbp]
  \resizebox{\textwidth}{16mm}{        
    \centering
    \begin{tabular}{|c|c|c|c|c|c|c|c|c|c|c|c|c|}
      \hline
      \multirow{2}*{Case \#}& 
      \multicolumn{3}{c|}{BdTgtLower}&
      \multicolumn{3}{c|}{MIPLower}&
      \multicolumn{3}{c|}{LPLower}&
      \multicolumn{3}{c|}{LPUpper}
      \cr\cline{2-13}
      &Time(s)& Var &Cons\#&Time(s)& Var\# &Cons\#&Time(s)& Var\# &Cons\#&Time(s)& Var\# &Cons\#
      \cr\hline
      5 &10.91 &44 &616 & 0.32 &69 &641 &0.43 &44 &616
      &0.52&615 &88
      \cr\hline
      9 &35.82 &152 &1364 &6.88 &189 &1401 &2.82 &152 &1364 & 2.95 &1347 &288
      \cr\hline
      14 &743.11 &377 &6532 &93.31 &458 & 6613 &468.55 &377 &6532&35.73 &6495 &718
      \cr\hline
      29 & 9436.28 & 1536 & 25477 & 3164.68 & 1703 & 25816 & 9236.42 & 1536 & 25477 & 483.72 & 24377 & 3126
      \cr\hline
      30 & 9888.07 & 1769 & 27176 & 3674.53 & 1934  & 27341 & 9484.48 & 1769 & 27176 & 531.31 & 27075 & 3438
      \cr\hline
      39 & 21764.43 & 3248 & 39265 & 9436.08 & 3771 & 41026 & 18329.58 & 3248 & 39265 & 1442.36 & 41342 & 4855
      \cr\hline
  \end{tabular}}
  \caption{  \label{tab:exp}
    Solution times, number of variables, and number of constraints for the procedures we considered on the Cases 5, 9, 14, 29, 30, and 39.
    }
    
\end{table}

\medskip
\begin{figure}[htp!]
\begin{center}
  \includegraphics[width=.45\linewidth-0.2mm]{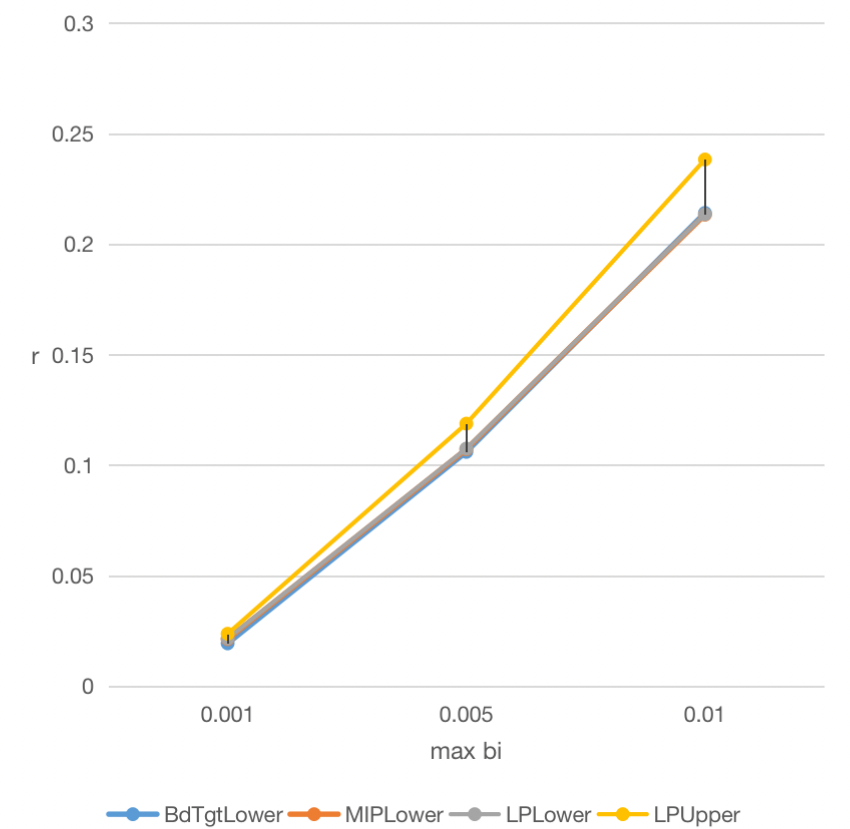}\hfill
  \includegraphics[width=.45\linewidth-0.2mm]{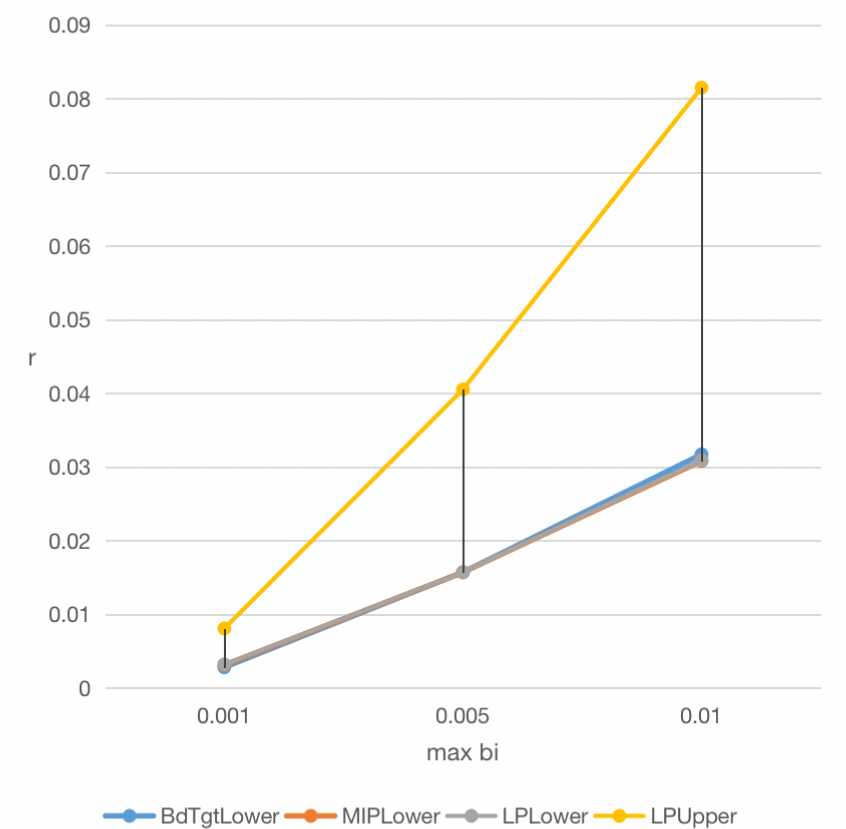}\\[0.5mm]
  \includegraphics[width=.45\linewidth-0.2mm]{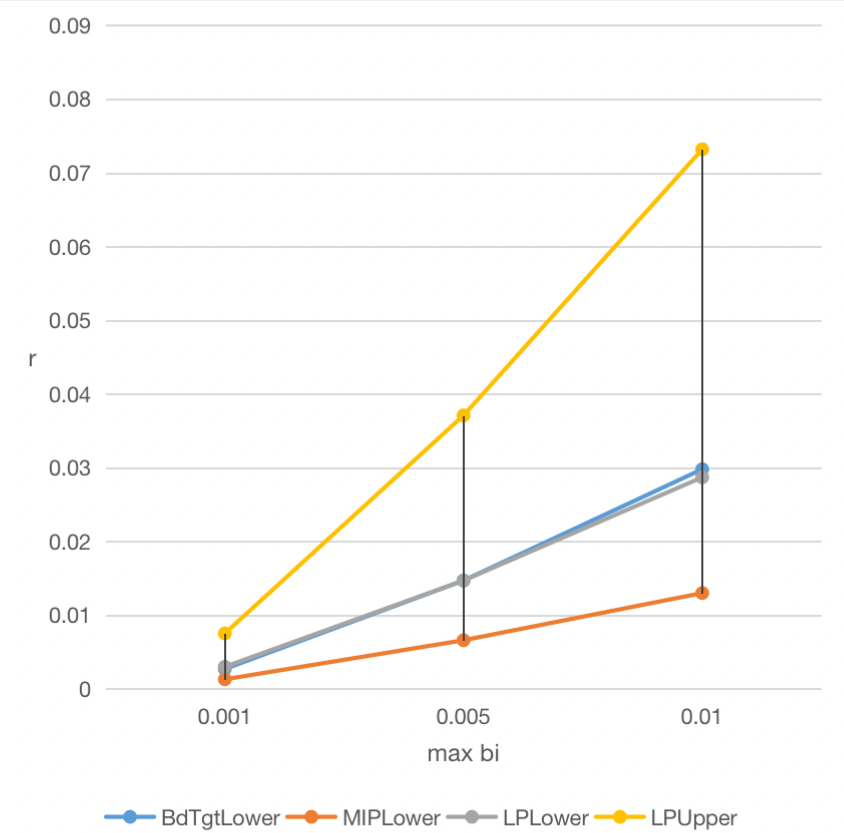}\hfill
  \includegraphics[width=.45\linewidth-0.2mm]{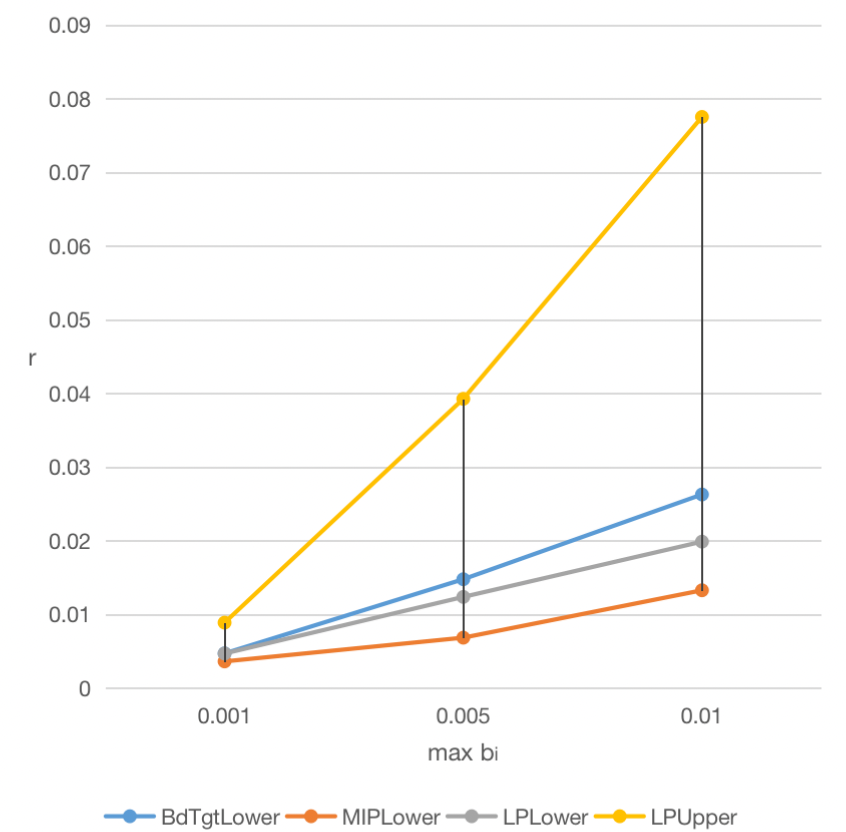}\\[0.5mm]
  \includegraphics[width=.45\linewidth-0.2mm]{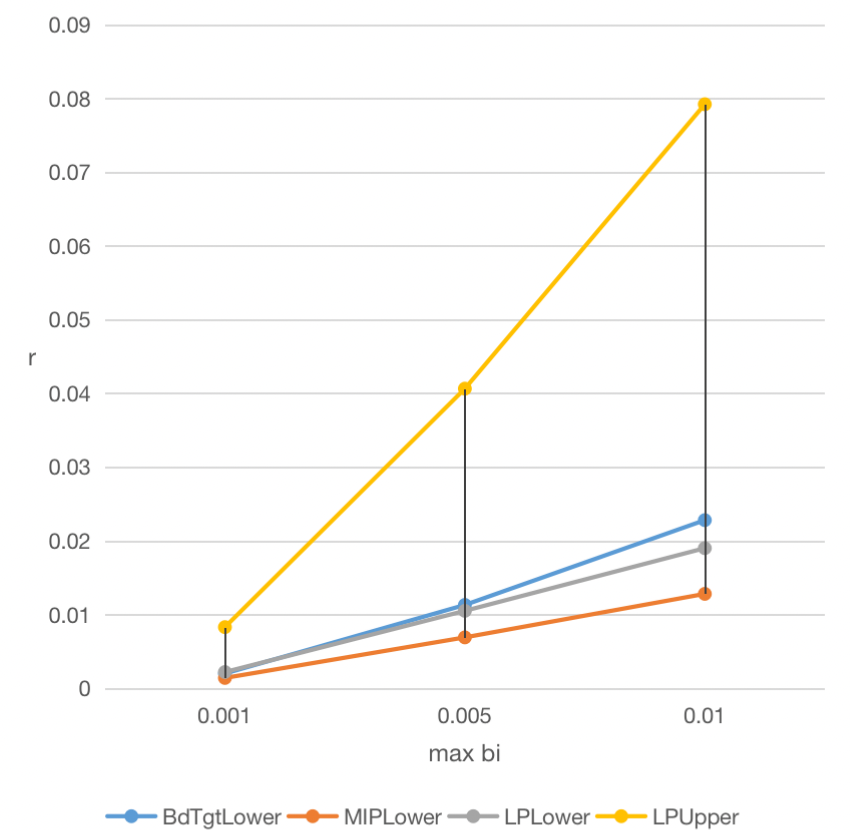}\hfill
  \includegraphics[width=.45\linewidth-0.2mm]{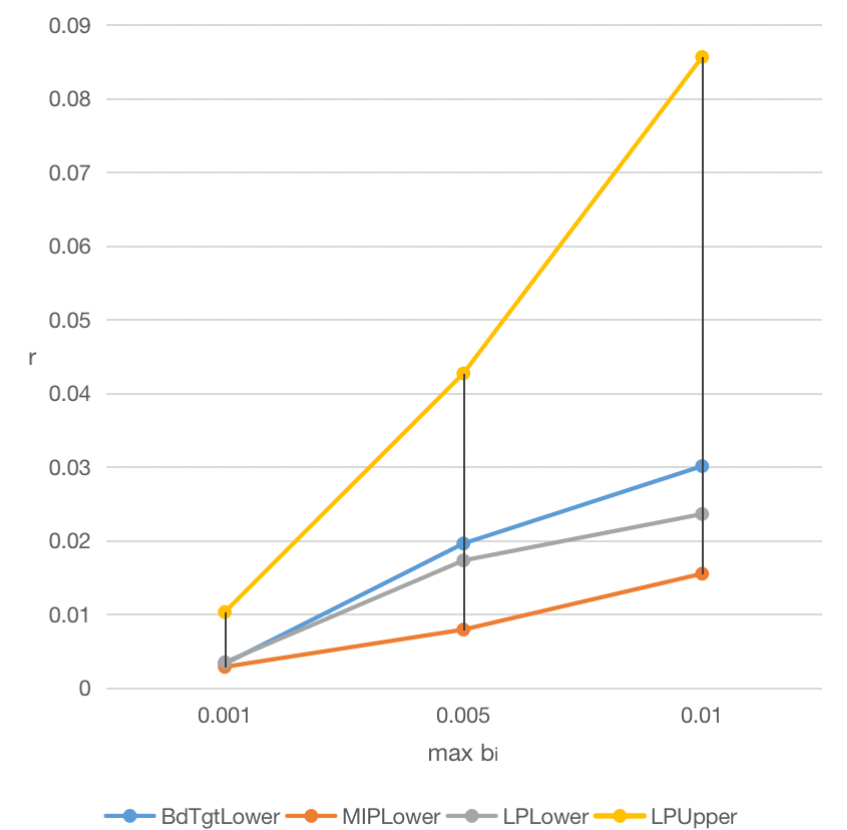}\\
\caption{Lower and upper bound estimates on the robustness margin for Case 5 (Top Left), Case 9 (Top Right), Case 14 (Middle Left), Case 29 (Middle Right), Case 30 (Bottom left), and Case 39 (Bottom right). }
\label{fig:Graphs1} 
\end{center}
\end{figure}

Case 57 is not included in the table and graphs as the only inner bound procedure to run in a reasonable time was the Feasibility procedure that produced a maximum robustness margin of $\approx 0.003$ with a computational time of $\approx 48924$ seconds, and a gap to the outer bound procedure of $\approx 0.030$ when $B=0.001$, which had a running time of $\approx 27080$ seconds.

As evident from the graphs, the LP Bound Tightening Procedure produces a better approximation of the lower bound on the robustness margin as the complexity of the data set increases. 
The outer bound procedure (in \cref{eq:OPTfeasOutRelaxb}) is derived using looser relaxations of the original problem compared to any of those used to derive lower bounds (in \cref{eq:OPTfeasrelaxLP1,eq:OPTfeasrelaxLP2,eq:OPTfeasrelaxLP3}).
  Hence we expect that the upper bound produced is not as tight as the lower bounds.
Certainly one would expect the bound tightening procedure to out perform the other inner bound procedures for all cases, but the choice of procedure parameters has a big effect on the efficiency and capability of the procedure. 
For instance, setting a low tolerance for a minimal sufficient change in the dimensions of $\vb$ will result in a better lower bound approximation, but an extremely long running time for most cases. 
Thus in the low as well as marginally high complexity cases, it should be expected that the other procedures will out perform the bound tightening procedure as these manually set parameters will have more of an impact.

\vspace*{-0.1in}
\section{Discussion} \label{sec:disc}


We have proposed novel and efficient techniques for determining robust solvability of quadratic systems with uncertainty. 
It is worth mentioning that the same machinery can be used to determine robustness margins of solutions to static systems by taking $A$ to be the identity, $e_i=0 \ \forall i$, and simply adjusting $\vb$ accordingly.

We have employed results on the computational complexity of QCQPs to shed some light on the hardness of the optimization problems in the key \cref{thm:RobFeas}.
It would be interesting to prove NP-hardness of the optimization problems in \cref{eq:OPTfeas} using direct arguments.

Our implementation on power systems in \cref{sec:numstd} is just one direct application of the general (theoretical) framework we have developed.
We are exploring other avenues for applying our framework including gas and water flow networks.

In the context of our application to power systems, we must point out that finding the robustness margin of an OPF instance is inherently harder than solving the original OPF instance itself.
Hence it is expected that larger sized instances of the default OPF problem are solved in practice than ones for which our inner and outer bound procedures are run efficiently.

We have presented (in \cref{ssec:compres}) bounds on the robustness margins for solutions of OPF instances.
One generalization we could consider is that of combining measures of robustness and optimality.
In practice, a near optimal solution with a large robustness margin might be more desirable than an optimal solution with a small robustness margin.


\vspace*{-0.1in}
\section*{Acknowledgment}

Krishnamoorthy, Luo, and Rapone acknowledge funding from the National Science Foundation through grants 1661348 and 1819229.
Dvijotham and Rapone were supported by the Pacific Northwest National Laboratory (PNNL) Center for Complex Systems Initiative while working on this project. 

\section*{Data Availability}

The datasets analyzed in the current study are available in the MATPOWER repository \cite{matpower} at \href{https://matpower.org/}{https://matpower.org/} and the NESTA repository \cite{CoGoSc2019}, the instances of which are available as part of the Library of IEEE PES Power Grid Benchmark repository at \href{https://github.com/power-grid-lib/}{https://github.com/power-grid-lib/}.
No new datasets were created as part of this study.


\begin{thebibliography}{10}

\bibitem{bandi2012tractable}
Chaithanya Bandi and Dimitris Bertsimas.
\newblock Tractable stochastic analysis in high dimensions via robust
  optimization.
\newblock {\em Mathematical Programming}, 134:23--70, 2012.

\bibitem{bandi2014optimal}
Chaithanya Bandi and Dimitris Bertsimas.
\newblock Optimal design for multi-item auctions: a robust optimization
  approach.
\newblock {\em Mathematics of Operations Research}, 39(4):1012--1038, 2014.

\bibitem{ben2009robust}
Aharon Ben-Tal, Laurent El~Ghaoui, and Arkadi Nemirovski.
\newblock {\em Robust optimization}.
\newblock Princeton University Press, 2009.

\bibitem{ben1999robust}
Aharon Ben-Tal and Arkadi Nemirovski.
\newblock Robust solutions of uncertain linear programs.
\newblock {\em Operations Research Letters}, 25(1):1--13, 1999.

\bibitem{BeEdMoPa2010}
Paul Bendich, Herbert Edelsbrunner, Dmitriy Morozov, and Amit Patel.
\newblock The robustness of level sets.
\newblock In {\em ESA 2010: 18th Annual European Symposium, Liverpool, UK,
  September 6-8, 2010. Proceedings, Part I}, pages 1--10. Springer Berlin
  Heidelberg, 2010.

\bibitem{BeNoTe2010}
Dimitris Bertsimas, Omid Nohadani, and Kwong~Meng Teo.
\newblock Nonconvex robust optimization for problems with constraints.
\newblock {\em INFORMS Journal on Computing}, 22(1):44--58, 2010.

\bibitem{bertsimas2004robust}
Dimitris Bertsimas, Dessislava Pachamanova, and Melvyn Sim.
\newblock Robust linear optimization under general norms.
\newblock {\em Operations Research Letters}, 32(6):510--516, 2004.

\bibitem{bertsimas2004price}
Dimitris Bertsimas and Melvyn Sim.
\newblock The price of robustness.
\newblock {\em Operations Research}, 52(1):35--53, 2004.

\bibitem{BiChHa2014}
Daniel Bienstock, Michael Chertkov, and Sean Harnett.
\newblock Chance-constrained optimal power flow: Risk-aware network control
  under uncertainty.
\newblock {\em SIAM Review}, 56(3):461--495, 2014.

\bibitem{BiMu2016}
Daniel Bienstock and Gonzalo Mu{\~n}oz.
\newblock {LP} approximations to mixed-integer polynomial optimization
  problems.
\newblock {\em CoRR}, abs/1501.00288v15, 2016.
\newblock
  \href{http://arxiv.org/abs/1501.00288}{http://arxiv.org/abs/1501.00288}.

\bibitem{bolognani2016existence}
Saverio Bolognani and Sandro Zampieri.
\newblock On the existence and linear approximation of the power flow solution
  in power distribution networks.
\newblock {\em IEEE Transactions on Power Systems}, 31(1):163--172, 2016.

\bibitem{ChSkPa2012}
Fr\'{e}d\'{e}ric Chazal, Primoz Skraba, and Amit Patel.
\newblock Computing well diagrams for vector fields on $\mathbb{R}^n$.
\newblock {\em Applied Mathematics Letters}, 25(11):1725--1728, 2012.

\bibitem{CoGoSc2019}
Carleton Coffrin, Dan Gordon, and Paul Scott.
\newblock {NESTA, The NICTA Energy System Test Case Archive}.
\newblock {\em ArXiv}, 2019.
\newblock \href{https://arxiv.org/abs/1411.0359}{arXiv:1411.0359}.

\bibitem{coffrin2015strengthening}
Carleton Coffrin, Hassan~L Hijazi, and Pascal Van~Hentenryck.
\newblock Strengthening convex relaxations with bound tightening for power
  network optimization.
\newblock In {\em International Conference on Principles and Practice of
  Constraint Programming}, pages 39--57. Springer International Publishing,
  2015.

\bibitem{DjTuritsyn}
Krishnamurthy Dvijotham, Hung Nguyen, and Konstantin Turitsyn.
\newblock Solvability regions of affinely parameterized quadratic equations.
\newblock {\em IEEE Control Systems Letters}, 2(1):25--30, Jan 2018.

\bibitem{EdMoPa2011}
Herbert Edelsbrunner, Dmitriy Morozov, and Amit Patel.
\newblock Quantifying transversality by measuring the robustness of
  intersections.
\newblock {\em Foundations of Computational Mathematics}, 11(3):345--361, 2011.

\bibitem{fonseca1995degree}
Irene Fonseca and Wilfrid Gangbo.
\newblock {\em Degree theory in analysis and applications}.
\newblock Number~2. Oxford University Press, 1995.

\bibitem{FrKr2015}
Peter Franek and Marek Kr{\v{c}}{\'a}l.
\newblock Robust satisfiability of systems of equations.
\newblock {\em Journal of the {ACM}}, 62(4):26:1--26:19, September 2015.

\bibitem{FrKr2016well}
Peter Franek and Marek Kr{\v{c}}{\'a}l.
\newblock On computability and triviality of well groups.
\newblock {\em Discrete and Computational Geometry}, 56(1):126--164, 2016.

\bibitem{FrKr2016pers}
Peter Franek and Marek Kr{\v{c}}{\'a}l.
\newblock Persistence of zero sets.
\newblock 2016.
\newblock Preprint available at
  \href{https://arxiv.org/abs/1507.04310}{https://arxiv.org/abs/1507.04310}.

\bibitem{FrKrWa2016}
Peter Franek, Marek Kr{\v{c}}{\'a}l, and Hubert Wagner.
\newblock Robustness of zero sets: {I}mplementation.
\newblock 2016.
\newblock Preprint available at
  \url{http://www.cs.cas.cz/~franek/rob-sat/experimental.pdf}.

\bibitem{FrRa2015}
Peter Franek and Stefan Ratschan.
\newblock Effective topological degree computation based on interval
  arithmetic.
\newblock {\em Mathematics of Computation}, 84:1265--1290, 2015.

\bibitem{FrHoLa2007}
Andreas Frommer, Fatmir Hoxha, and Bruno Lang.
\newblock Proving the existence of zeros using the topological degree and
  interval arithmetic.
\newblock {\em {Journal of Computational and Applied Mathematics}},
  199:397--402, 2007.

\bibitem{FrLa2005}
Andreas Frommer and Bruno Lang.
\newblock Existence tests for solutions of nonlinear equations using {B}orsuk's
  theorem.
\newblock {\em SIAM Journal on Numerical Analysis}, 43(3):1348--1361, 2005.

\bibitem{GRIEWANK2014}
Andreas Griewank.
\newblock {On Automatic Differentiation and Algorithmic Linearization}.
\newblock {\em {Pesquisa Operacional}}, 34:621 -- 645, 12 2014.

\bibitem{kocuk2016strong}
Burak Kocuk, Santanu~S. Dey, and Andy~X. Sun.
\newblock {Strong SOCP relaxations for the optimal power flow problem}.
\newblock {\em Operations Research}, 64(6):1177--1196, 2016.

\bibitem{Lasserre2006}
Jean~B Lasserre.
\newblock Robust global optimization with polynomials.
\newblock {\em Mathematical programming}, 107(1-2):275--293, 2006.

\bibitem{Lasserre2011}
Jean~B. Lasserre.
\newblock Min-max and robust polynomial optimization.
\newblock {\em Journal of Global Optimization}, 51:1--10, 2011.

\bibitem{lotfi2022data}
Reza Lotfi, Bahareh Kargar, Alireza Gharehbaghi, Mohamad Afshar, Mohammad~Sadra
  Rajabi, and Nooshin Mardani.
\newblock A data-driven robust optimization for multi-objective renewable
  energy location by considering risk.
\newblock {\em Environment, Development and Sustainability}, pages 1--22, 2022.

\bibitem{Ma1973}
Sibe Marde{\v{s}}i{\'c}.
\newblock Compact subsets of $\mathbb{R}^n$ and dimension of their projections.
\newblock {\em Proceedings of the American Mathematical Society},
  41(2):631--633, 1973.

\bibitem{mathieu2011examining}
Johanna~L Mathieu, Duncan~S Callaway, and Sila Kiliccote.
\newblock Examining uncertainty in demand response baseline models and
  variability in automated responses to dynamic pricing.
\newblock In {\em Decision and Control and European Control Conference
  (CDC-ECC), 2011 50th IEEE Conference on}, pages 4332--4339. IEEE, 2011.

\bibitem{morales2007point}
Juan~M Morales and Juan Perez-Ruiz.
\newblock Point estimate schemes to solve the probabilistic power flow.
\newblock {\em IEEE Transactions on Power Systems}, 22(4):1594--1601, 2007.

\bibitem{MoVrYa2002}
Bernard Mourrain, Michael~N. Vrahatis, and Jean-Claude Yakoubsohn.
\newblock On the complexity of isolating real roots and computing with
  certainty the topological degree.
\newblock {\em {Journal of Complexity}}, 18(2):612--640, 2002.

\bibitem{OrChCh2006}
Donal {O'Regan}, Yeol~Je Cho, and Yu-Qing Chen.
\newblock {\em Topological Degree Theory and Applications}, volume~10.
\newblock Chapman and Hall/CRC, Boca Raton, FL, 2006.

\bibitem{PaBo2017}
Jaehyun Park and Stephen Boyd.
\newblock General heuristics for nonconvex quadratically constrained quadratic
  programming.
\newblock {\em CoRR}, abs/1703.07870, 2017.

\bibitem{RoVrOlAn2015}
Line Roald, Maria Vrakopoulou, Frauke Oldewurtel, and G{\"o}ran Andersson.
\newblock Risk-based optimal power flow with probabilistic guarantees.
\newblock {\em International Journal of Electrical Power and Energy Systems},
  72:66--74, 2015.
\newblock The Special Issue for 18th Power Systems Computation Conference.

\bibitem{taylor2015uncertainty}
Joshua~A Taylor and Johanna~L Mathieu.
\newblock Uncertainty in demand response—identification, estimation, and
  learning.
\newblock In {\em The Operations Research Revolution}, pages 56--70. INFORMS,
  2015.

\bibitem{TsBiTa2016}
Shih-Hao Tseng, Eilyan Bitar, and Ao~Tang.
\newblock Random convex approximations of ambiguous chance constrained
  programs.
\newblock In {\em 2016 IEEE 55th Conference on Decision and Control (CDC)},
  pages 6210--6215, Dec 2016.

\bibitem{VaBo1996}
Lieven Vandenberghe and Stephen Boyd.
\newblock Semidefinite programming.
\newblock {\em SIAM review}, 38(1):49--95, 1996.

\bibitem{EPFLA}
Cong Wang, Andrey Bernstein, Jean-Yves Le~Boudec, and Mario Paolone.
\newblock {Existence and Uniqueness of Load-Flow Solutions in Three-Phase
  Distribution Networks}.
\newblock {\em IEEE Transactions on Power Systems}, 32(4):3319--3320, 2017.

\bibitem{EPFLB}
Cong Wang, Andrey Bernstein, Jean-Yves Le~Boudec, and Mario Paolone.
\newblock {Explicit Conditions on Existence and Uniqueness of Load-Flow
  Solutions in Distribution Networks}.
\newblock {\em IEEE Transactions on Smart Grid}, 9(2):953--962, 2018.

\bibitem{wang1992interval}
Zian Wang and Fernando~L Alvarado.
\newblock Interval arithmetic in power flow analysis.
\newblock {\em IEEE Transactions on Power Systems}, 7(3):1341--1349, 1992.

\bibitem{zhang2011chance}
Hui Zhang and Pu~Li.
\newblock Chance constrained programming for optimal power flow under
  uncertainty.
\newblock {\em IEEE Transactions on Power Systems}, 26(4):2417--2424, 2011.

\bibitem{matpower}
Ray~Daniel Zimmerman, Carlos~Edmundo Murillo-Sanchez, and Robert~John Thomas.
\newblock Matpower: Steady-state operations, planning, and analysis tools for
  power systems research and education.
\newblock {\em IEEE Transactions on Power Systems}, 26(1):12--19, Feb 2011.

\end{thebibliography}

\end{document}